\documentclass[12pt,a4paper]{article}
\usepackage{amsmath,amsfonts,amssymb,amsthm,amscd}
\newtheorem{thm}{Theorem}

\newtheorem{prop}[thm]{Proposition}

\newtheorem{cor}[thm]{Corollary}

\newtheorem{notation}[thm]{Notations}

\begin{document}

\author{Liana David}

\title{A prolongation of the conformal-Killing operator
on quaternionic-K\"{a}hler manifolds}

\maketitle

\textbf{Abstract:} A $2$-form on a quaternionic-K\"{a}hler
manifold $(M, g)$ of dimension $4n\geq 8$ is called compatible
(with the quaternionic structure) if it is a section of the direct
sum bundle $S^{2}H\oplus S^{2}E.$ We construct a prolongation
$\mathcal D$ of the conformal-Killing operator acting on
compatible $2$-forms. We show that $\mathcal D$ is flat if and
only if the quaternionic-Weyl tensor of $(M, g)$ vanishes.
Consequences of this result are developed. We construct a
skew-symmetric multiplication on the space of conformal-Killing
$2$-forms on $(M, g)$ and we study its properties in connection
with the subspace of compatible conformal-Killing $2$-forms.\\



{\it Author's address:} Institute of Mathematics "Simion Stoilow"
of the Romanian Academy; Calea Grivitei nr 21, Bucharest, Romania;
liana.david@imar.ro; liana.r.david@gmail.com; tel.
0040-21-3196506; fax 0040-21-3196505.

\section{Introduction}

Let $(M^{m}, g)$ be a Riemannian $m$-dimensional manifold. For any
$1\leq p\leq m$ consider the tensor product bundle $T^{*}M\otimes
\Lambda^{p}(M)$ and its irreducible $O(m)$-decomposition:
\begin{equation}\label{tm}
T^{*}M\otimes \Lambda^{p}(M) = \Lambda^{p+1}(M)\oplus
\Lambda^{p-1}(M)\oplus {\mathcal T}^{p,1}(M),
\end{equation}
where the sub-bundle ${\mathcal T}^{p,1}(M)$ of $T^{*}M\otimes
\Lambda^{p}(M)$ is the intersection of the kernels of the wedge
product and inner contraction maps. If $\psi\in\Omega^{p}(M)$ is a
$p$-form, the covariant derivative $\nabla \psi$ with respect to
the Levi-Civita connection $\nabla$ of $g$ is a section of
$T^{*}M\otimes\Lambda^{p}(M)$ and its projection onto
$\Lambda^{p+1}(M)$ and $\Lambda^{p-1}(M)$ is given, essentially,
by the exterior derivative $d\psi$ and the codifferential $\delta
\psi$, respectively. The $p$-form $\psi$ is called
conformal-Killing if the projection of $\nabla\psi$ onto the third
component in the decomposition (\ref{tm}) is trivial;
equivalently, if the conformal-Killing equation
\begin{equation}\label{conformalkil}
\nabla_{Y} \psi  =\frac{1}{p+1} i_{Y}d\psi  -\frac{1}{m-p+1}Y\land
\delta \psi ,\quad\forall Y\in TM
\end{equation}
is satisfied. (Above and often in this note we identify, without
mentioning explicitly, tangent vectors with $1$-forms by means of
the Riemannian duality). A co-closed conformal-Killing form is
called Killing. There is a large literature on conformal-Killing
forms (for a survey, see e.g. \cite{sem}). A $1$-form is
(conformal)-Killing if the dual vector field is a
(conformal)-Killing vector field. Conformal-Killing forms exist on
spaces of constant curvature, on Sasaki manifolds, on some classes
of K\"{a}hler manifolds (like Bochner-flat or conformally-Einstein
K\"{a}hler manifolds) and on Riemannian manifolds which admit
Killing spinors \cite{sem}. On a compact quaternionic-K\"{a}hler
manifold of dimension at least eight, any Killing $p$-form ($p\geq
2$) is parallel (see \cite{quatk}) and any conformal-Killing
$2$-form parallel unless the quaternionic-K\"{a}hler manifold is
isometric to the standard quaternionic projective space, in which
case the codifferential defines an isomorphism from the space of
conformal-Killing $2$-forms to the space of Killing vector fields
(see \cite{lianamax}).

This paper is concerned with conformal-Killing $2$-forms on (not
necessarily compact) quaternionic-K\"{a}hler manifolds. In Section
\ref{0}, devoted to fix notations, we recall basic facts from
quaternionic-K\"{a}hler geometry.

Section \ref{2} contains the main result of our paper (see Theorem
\ref{main2}). A $2$-form on a quaternionic-K\"{a}hler manifold
$(M, g)$ (always assumed to be connected and of dimension $4n\geq
8$) is called compatible (with the quaternionic structure) if it
is a section of the direct sum bundle $S^{2}H\oplus S^{2}E$. We
find a connection $\mathcal D$ on the bundle $S^{2}H\oplus
S^{2}E\oplus TM$, which is a prolongation of the conformal-Killing
operator acting on compatible $2$-forms (i.e. the space of
$\mathcal D$-parallel sections is isomorphic to the space
${\mathcal C}_{2}(M)$ of compatible conformal-Killing $2$-forms on
$(M,g)$). We compute the curvature of $\mathcal D$ and we show
that $\mathcal D$ is flat if and only if the quaternionic-Weyl
tensor $W^{Q}$ of $(M, g)$ is zero.

Section \ref{applications} is devoted to applications of Theorem
\ref{main2}. First, we prove that the dimension of the vector
space ${\mathcal C}_{2}(M)$ is less or equal to $(n+1)(2n+3)$ and
equality holds on the standard quaternionic projective space
$\mathbb{H}P^{n}$ (see Corollary \ref{dim}). Next, we assume that
the quaternionic-K\"{a}hler manifold $(M, g)$ has non-zero scalar
curvature and, under this additional assumption, we prove that if
$(M, g)$ admits a non-parallel compatible conformal-Killing
$2$-form then the holonomy group of $(M, g)$ coincides with
$\mathrm{Sp}(1)\mathrm{Sp}(n)$ (see Proposition \ref{sp}). In
particular, there are no non-parallel compatible conformal-Killing
$2$-forms on open subsets of Wolf spaces non-isomorphic to the
standard quaternionic projective space. At the end of this section
we prove that a compatible conformal-Killing $2$-form $\psi$ on a
quaternionic-K\"{a}hler manifold $(M,g)$ of non-zero scalar
curvature is determined by its $S^{2}E$-part $\psi^{S^{2}E}$ and
we find the image and the inverse of the map ${\mathcal
C}_{2}(M)\ni\psi \rightarrow \psi^{S^{2}E}$ (see Proposition
\ref{hamkil}).

Finally, in Section \ref{bracket-section} we construct a
skew-symmetric multiplication $[\cdot , \cdot ]$ on the space of
conformal-Killing $2$-forms on a quaternionic-K\"{a}hler manifold
$(M, g)$, which preserves the subspace ${\mathcal C}_{2}(M)$ of
compatible conformal-Killing $2$-forms. When the scalar curvature
of $(M, g)$ is non-zero, $({\mathcal C}_{2}(M), [\cdot , \cdot ])$
is a Lie algebra and the codifferential $\delta : {\mathcal
C}_{2}(M)\rightarrow \mathrm{isom}(M,g)$ is a Lie algebra
homomorphism (see Corollary \ref{particular}).

\section{Quaternionic-K\"{a}hler manifolds}\label{0}

In this Section we recall basic definitions and results from
quaternionic-K\"{a}hler geometry, which will be useful in our
treatment of conformal-Killing forms.

\subsection{Basic definitions}

A quaternionic-K\"{a}hler manifold is a Riemannian manifold $(M,
g)$ of dimension $4n\geq 8$ with holonomy group included in
$\mathrm{Sp}(1)\mathrm{Sp}(n).$ Equivalently, there is a rank
three vector sub-bundle $Q\subset\mathrm{End}(TM)$ preserved by
the Levi-Civita connection $\nabla$ of $g$ and locally generated
by three anti-commuting almost complex structures $\{ J_{1},
J_{2}, J_{3}\}$ with $J_{3}= J_{1}J_{2}.$ Such a system of almost
complex structures is usually called a local admissible basis of
$Q$. The metric $g$ is Einstein; moreover, $g$ is Ricci-flat if
and only if $(M, g)$ is locally hyper-K\"{a}hler, i.e. in a
neighborhood of any point there is an admissible basis of $Q$
formed by $\nabla$-parallel complex structures.

The curvature of $g$ has the expression
\begin{equation}\label{rg} R^{g}_{X,
Y}=-\frac{\nu}{4}\left( X\land Y +\sum_{i=1}^{3}J_{i}X\land J_{i}Y
+2\sum_{i=1}^{3}\omega_{i}(X, Y)\omega_{i}\right) +W^{Q}_{X, Y}
\end{equation}
where $\nu :=\frac{k}{4n(n+2)}$ is the reduced scalar curvature
($k$ being the usual scalar curvature), $\{ J_{1}, J_{2}, J_{3}\}$
is a local admissible basis of $Q$, with K\"{a}hler forms
$\omega_{i}:= g(J_{i}\cdot , \cdot )$, and $W^{Q}$ is the
quaternionic-Weyl tensor, which satisfies
$$
W^{Q}_{JX, JY}= W^{Q}_{X, Y},\quad\forall X, Y\in TM,\quad\forall
J\in Q, \quad J^{2}=-\mathrm{Id}
$$
and belongs to the kernel of the Ricci contraction, i.e.
$$
\sum_{k=1}^{4n}W^{Q}_{X, e_{k}}(e_{k}) =0,\quad\forall X\in TM
$$
where  $\{ e_{k}\}$ is a local orthonormal frame of $TM$  (above
"$\mathrm{Id}$" denotes the identity endomorphism).\\

Let $E$ and $H$ be the (locally defined) complex vector bundles
over $M$, of rank $2n$ and $2$ respectively, associated to the
standard representations of $\mathrm{Sp}(n)$ and $\mathrm{Sp}(1)$
on $\mathbb{C}^{2n}$ and $\mathbb{C}^{2}$. The bundles $E$ and $H$
play the role of spin bundles in conformal geometry, since
\begin{equation}\label{eh}
T_{\mathbb{C}}M = E\otimes H.
\end{equation}
Let $\omega_{E}\in
\Lambda^{2}(E^{*})$ and $\omega_{H}\in \Lambda^{2} (H^{*})$
be the complex symplectic forms on $E$ and $H$,
defined by the standard $\mathrm{Sp}(n)$-invariant
symplectic form of $\mathbb{C}^{2n}$, and, respectively, the standard
$\mathrm{Sp}(1)$-invariant symplectic form of $\mathbb{C}^{2}$.
We shall identify $E$ with $E^{*}$ by means of the isomorphism
$E\ni e\rightarrow \omega_{E}(e,\cdot )$, and, similarly,
$H$ with $H^{*}$ using $\omega_{H}.$ At the level of $2$-forms,
(\ref{eh}) induces
a decomposition
\begin{equation}\label{lambda2}
\Lambda^{2}(T_{\mathbb{C}}M) = S^{2}H\oplus S^{2}E\oplus
S^{2}H\otimes \Lambda^{2}_{0}E,
\end{equation}
where $\Lambda^{2}_{0}E\subset\Lambda^{2}E$ is the kernel of the
natural contraction with $\omega_{E}.$  The bundles $S^{2}H$ and
$S^{2}E$ are the complexifications of the bundle $Q$ and,
respectively, of the bundle of $Q$-Hermitian $2$-forms (a $2$-form
is $Q$-Hermitian if it is of type $(1,1)$ with respect to any
compatible almost complex structure; an almost complex structure
is compatible if it is a section of the bundle $Q$). We shall
often identify, implicitly,  real bundles with their
complexifications. The projections of a $2$-form
$\psi\in\Omega^{2}(M)$ on $S^{2}H$ and $S^{2}E$ have the following
expressions:
\begin{equation}\label{ad-s2}
\psi^{S^{2}H}=\frac{1}{4n}\sum_{i,k}\psi (e_{k},
J_{i}e_{k})\omega_{i}
\end{equation}
\begin{equation}
\psi^{S^{2}{E}}= \frac{1}{4}\left( \psi +\sum_{i=1}^{3} \psi
(J_{i}\cdot ,J_{i}\cdot ) \right) ,
\end{equation}
where $\{ J_{1}, J_{2}, J_{3}\}$ is an admissible basis of $Q$,
with associated K\"{a}hler forms $\omega_{i}$, and $\{ e_{k}\}$ is
a local orthonormal frame of $TM.$

\subsection{The Penrose operator}\label{sectiunepen}

Penrose (or twistor) operators appear in the literature on several
classes of manifolds. On a conformal $4$-manifold a Weyl
connection defines a Penrose operator and Penrose operators
obtained in this way were studied in \cite{paul}, \cite{gaud3}.
Similarly, any quaternionic connection on a quaternionic manifold
defines a Penrose operator (see \cite{liana}). An important role
in our treatment of compatible conformal-Killing $2$-forms will be
played by the Penrose operator on a quaternionic-K\"{a}hler
manifold $(M, g)$, defined using the Levi-Civita connection (see
\cite{sal}, also \cite{amp}). It is a first order differential
operator
$$
\bar{D}:\Gamma (S^{2}H)\rightarrow \Gamma (S^{3}H\otimes E)
$$
whose value on a section $\sigma$ of $S^{2}H$ is obtained by
taking the covariant derivative $\nabla\sigma$ (where $\nabla$ is
the Levi-Civita connection), which is a section of $T^{*}_{\mathbb
C}M\otimes S^{2}H$ and projecting it onto the second component of
the irreducible decomposition
\begin{equation}\label{decomposition}
T^{*}_{\mathbb C}M\otimes S^{2}H\cong E\otimes H\oplus  E\otimes S^{3}H.
\end{equation}
Let $\{J_{1}, J_{2}, J_{3}\}$ be a local admissible
basis of $Q$, with K\"{a}hler forms $\omega_{1}$, $\omega_{2}$ and
$\omega_{3}$ respectively. According
to Section 5 of \cite{amp}, for any section $\sigma$ of $Q$,
\begin{equation}\label{sigma}
\bar{D}\sigma = \nabla\sigma +\frac{1}{3}\sum_{i=1}^{3} (\delta
\sigma )\circ J_{i}\otimes\omega_{i}.
\end{equation}
The (real) solutions of the twistor equation (i.e. sections
of the quaternionic bundle $Q$ which belong to the kernel of
$\bar{D}$) can be easily described when the reduced scalar
curvature $\nu \neq 0$. In this case the map
\begin{equation}\label{bijectiva}
\mathrm{isom}(M, g)\ni X\rightarrow \frac{2}{3\nu} (\nabla
X)^{S^{2}H}\in \mathrm{Ker}\bar{D}
\end{equation}
is an isomorphism, with inverse the codifferential (see
\cite{amp},  \cite{sal}).

\subsection{Killing and quaternionic vector fields}\label{adaugat}

There is a useful criterion to check when a vector field on a
quaternionic-K\"{a}hler manifold of non-zero scalar curvature is
Killing (see \cite{brambu}). We shall use this criterion in
Section \ref{applications}. It is stated as follows.

Recall first that a vector field, say $X$, on a
quaternionic-K\"{a}hler manifold $(M^{4n}, g)$ (with $4n\geq 8$)
is quaternionic, if its flow preserves the quaternionic bundle
$Q$, or, equivalently, $[\nabla X, Q]\subset Q$. The criterion
states that a vector field which is quaternionic and
divergence-free is necessarily Killing. This is a consequence the
theory developed in \cite{brambu} and \cite{brambu2}. For
completeness of our exposition, we include the argument, which
goes as follows. Let $X$ be a quaternionic vector field on $(M,
g)$. From \cite{brambu}, page 303, and \cite{brambu2},
$L_{X}\nabla = S^{\alpha},$ where $\alpha\in\Omega^{1}(M)$ is
defined by
$$
\alpha = df,\quad f :=\frac{1}{4(n+1)} \mathrm{Trace}(\nabla X).
$$
and $S^{\alpha}$ is a $1$-form with values in $\mathrm{End}(TM)$
defined by
\begin{equation}\label{a2}
S_{Z}^{\alpha}:=\alpha (Z)\mathrm{Id}_{TM}+\alpha\otimes Z
-\sum_{i=1}^{3}\left(\alpha (J_{i}Z)J_{i}+ (\alpha\circ
J_{i})\otimes J_{i}Z\right) ,\quad \forall Z\in TM.
\end{equation}
Now, if $X$ is also divergence-free, then $\alpha =0$ and
$L_{X}\nabla =0$. This implies that $X$ preserves the curvature of
$g$ (viewed as a $2$-form with values in $\mathrm{End}(TM)$) and
also the Ricci tensor (viewed as a bilinear form on $TM$). Since
$g$ is Einstein with non-zero scalar curvature, $X$ is Killing.

\begin{notation} {\rm
Let $(M, g)$ be a quaternionic-K\"{a}hler manifold of dimension $4n\geq 8.$
We shall use the scalar product $\langle\cdot , \cdot \rangle$ on
$\Lambda^{2}(M)$, defined by
\begin{equation}\label{sc}
\langle X\wedge Y, Z\wedge V\rangle = g(X, Z) g(Y, V) - g(X, V)
g(Y, Z).
\end{equation}
With respect to the scalar product (\ref{sc}), the K\"{a}hler forms
$\{ \omega_{1}, \omega_{2}, \omega_{3}\}$ corresponding to a
local admissible
basis $\{ J_{1}, J_{2}, J_{3}\}$ of the quaternionic bundle
$Q$ are orthogonal and
$$
| \omega_{1}|^{2} = |\omega_{2}|^{2} = |\omega_{3}|^{2} = 2n.
$$
In our conventions, $\nabla$ will always denote the Levi-Civita
connection of $g$.}\end{notation}

\section{Our main result}\label{2}

Given a linear differential operator $D$, it is sometimes useful
to determine a vector bundle
connection (called a prolongation of $D$) whose space of parallel
sections is isomorphic with the kernel of $D$.
In general, there are several connections with this property.
However, if one prolongation is flat, then all are.

This Section contains the main result of this paper -  we
determine a prolongation $\mathcal D$ of the conformal-Killing
operator acting on compatible $2$-forms on a
quaternionic-K\"{a}hler manifold $(M,g)$ and we show that
$\mathcal D$ is flat if and only if the quaternionic-Weyl tensor
$W^{Q}$ of $(M, g)$ is zero. The prolongation $\mathcal D$ acts on
the direct sum bundle $S^{2}H\oplus S^{2}E\oplus TM.$ More
precisely, we state:

\begin{thm}\label{main2}
Let $(M, g)$ be a quaternionic-K\"{a}hler manifold of dimension
$4n\geq 8$, reduced scalar curvature $\nu$ and quaternionic-Weyl
tensor $W^{Q}.$ Define a connection ${\mathcal D}$ on
$S^{2}H\oplus S^{2}E\oplus TM$, by
\begin{align*}
&{\mathcal D}_{Z}(\psi, X)^{S^{2}E\oplus S^{2}H} =\nabla_{Z}\psi-
\frac{1}{4n-1}\left( X\wedge
Z+\sum_{i=1}^{3}J_{i}X\wedge J_{i}Z -\sum_{i=1}^{3}\omega_{i}(X, Z)\omega_{i}\right)\\
&{\mathcal D}_{Z}(\psi ,X)^{TM} =\nabla_{Z}X - \frac{4n-1}{4}
i_{Z}\left( \nu \psi^{S^{2}E} -2\nu \psi^{S^{2}H}
+\frac{1}{n+1}W^{Q}(\psi)\right) ,\\
\end{align*}
where $\{ J_{1}, J_{2}, J_{3}\}$ is a local admissible basis of
the quaternionic bundle $Q$, with K\"{a}hler forms $\omega_{1},
\omega_{2}, \omega_{3}$, $\psi$ is a section of $S^{2}H\oplus
S^{2}E$ and $X$, $Z$ are vector fields on $M$. Then $\mathcal D$
is a prolongation of the conformal-Killing operator acting on
compatible $2$-forms. Moreover, $\mathcal D$ is flat if and only
if $W^{Q}=0.$
\end{thm}

We divide the proof of Theorem \ref{main2} into three steps. In a
first stage, we rewrite the conformal-Killing equation on
compatible $2$-forms in a way suitable for the prolongation
procedure (see Proposition \ref{corolar}). We remark that
Proposition \ref{corolar} has already been proved in
\cite{lianamax} in the compact case. We now adapt the argument
also to the case when the quaternionic-K\"{a}hler manifold is
non-compact. In a second stage, we show that the connection
$\mathcal D$ from Theorem \ref{main2} is a prolongation of the
conformal-Killing operator acting on compatible $2$-forms (see
Proposition \ref{p1}). Finally, we compute the curvature of
$\mathcal D$ and we show that $\mathcal D$ is flat if and only if
$W^{Q}=0$ (see Proposition \ref{curbura}). Details are as follows.

\begin{prop}\label{corolar} A compatible $2$-form $\psi$ on $(M, g)$ is conformal-Killing if
and only if it satisfies
\begin{equation}\label{echivalent}
\nabla_{Y}\psi = \frac{1}{4n-1}\left( X\wedge Y
+\sum_{i=1}^{3}J_{i}X\wedge J_{i}Y -\sum_{i=1}^{3}\omega_{i}(X,
Y)\omega_{i} \right) ,\quad\forall Y\in TM ,
\end{equation}
where $\{ J_{1}, J_{2}, J_{3}\}$ is a local admissible basis of
the quaternionic bundle $Q$, $\omega_{1},\omega_{2},\omega_{3}$
are the associated K\"{a}hler forms and $X$ is a vector field
(necessarily equal to $\delta\psi$). In particular,
$\psi^{S^{2}H}$ satisfies the twistor equation.
\end{prop}

\begin{proof} Let $\psi$ be a compatible
conformal-Killing $2$-form on
$(M, g)$ and $X = \delta \psi$ its codifferential. Hence $\psi$
satisfies
\begin{equation}\label{confkil}
\nabla_{Y}\psi = \frac{1}{3}i_{Y}d\psi +\frac{1}{4n-1}X \wedge
Y,\quad \forall Y\in TM.
\end{equation}

Projecting equation (\ref{confkil}) onto $S^{2}H$ we obtain
\begin{equation}\label{s2h}
\nabla_{Y} \psi^{S^{2}H} = \frac{1}{3} (i_{Y}d\psi
)^{S^{2}H}+\frac{1}{2n(4n-1)} \sum_{i=1}^{3} \omega_{i}(X, Y)
\omega_{i}.
\end{equation}
Let $\{ e_{k}\}$ be a local orthonormal frame of $TM$. To simplify
notations, sometimes we omit below the summation sign over $1\leq
k\leq 4n.$ Note that
\begin{align*}
(i_{Y}d\psi)^{S^{2}H}& =\frac{1}{4n}\sum_{i=1}^{3} (d\psi)(Y,
e_{k},
J_{i}e_{k})\omega_{i}\\
&=\frac{1}{4n}\sum_{i=1}^{3}\left( (\nabla_{Y}\psi )(e_{k},
J_{i}e_{k}) -(\nabla_{e_{k}}\psi ) (Y, J_{i}e_{k})
+(\nabla_{J_{i}e_{k}}\psi )(Y, e_{k})\right) \omega_{i}\\
&=\nabla_{Y}\psi^{S^{2}H}+\frac{1}{4n}\sum_{i=1}^{3}\left(
(\nabla_{J_{i}e_{k}}\psi )(Y, e_{k}) -(\nabla_{e_{k}}\psi )(Y,
J_{i}e_{k})\right)\omega_{i} .
\end{align*}
Define
\begin{equation}\label{definitia}
E(\psi , Y) :=\frac{1}{4n}\sum_{i=1}^{3}\left(
(\nabla_{J_{i}e_{k}}\psi )(Y, e_{k}) -(\nabla_{e_{k}}\psi )(Y,
J_{i}e_{k})\right)\omega_{i} .
\end{equation}
With this notation, equation (\ref{s2h}) becomes
\begin{equation}\label{ADD}
\nabla_{Y}\psi^{S^{2}H} = \frac{1}{2} E(\psi ,Y) +\frac{3}{4n(4n-1)}
\sum_{i=1}^{3}\omega_{i}(X, Y)\omega_{i}.
\end{equation}
We now compute $E(\psi , Y).$ It is easy to check, using that
$\psi^{S^{2}E}$ is of type $(1,1)$ with respect to any
compatible complex structure, that
\begin{equation}\label{e1}
E(\psi^{S^{2}E}, Y)=
\frac{1}{2n}\sum_{i=1}^{3}(\delta\psi^{S^{2}E})(J_{i}Y)
\omega_{i}.
\end{equation}
On the other hand, for any fixed $i\in \{ 1,2,3\}$,
\begin{align*}
&\sum_{k=1}^{4n}\left( (\nabla_{J_{i}e_{k}}\psi^{S^{2}H} )(Y, e_{k}) -
(\nabla_{e_{k}} \psi^{S^{2}H})(Y, J_{i}e_{k})\right)\\
&=   \sum_{k=1}^{4n}\left( \langle \nabla_{J_{i}e_{k}}\psi ,
(Y\wedge e_{k})^{S^{2}H}
\rangle  - \langle \nabla_{e_{k}} \psi ,
(Y\wedge J_{i}e_{k})^{S^{2}H}\rangle\right)\\
&=\frac{1}{n} \sum_{j=1}^{3}\langle
\nabla_{J_{i}J_{j}Y}\psi^{S^{2}H}, \omega_{j}\rangle ,
\end{align*}
where we used (\ref{ad-s2}). From this and (\ref{definitia}) it
follows  that
\begin{equation}\label{e-1}
E(\psi^{S^{2}H},Y) = -\frac{1}{4n^{2}} \sum_{i,j}\langle
\nabla_{J_{j}J_{i}Y}\psi^{S^{2}H},\omega_{j}\rangle \omega_{i}-\frac{1}{n}
\nabla_{Y}\psi^{S^{2}H}.
\end{equation}
From (\ref{e1}) and (\ref{e-1}), relation (\ref{ADD}) becomes
\begin{equation}\label{wanted}
\nabla_{Y}\psi^{S^{2}H} = \sum_{i=1}^{3}\alpha (J_{i}Y) \omega_{i}
\end{equation}
where the $1$-form $\alpha\in\Omega^{1}(M)$ is given by
\begin{equation}\label{alpha}
\alpha (Y):= \frac{1}{2(2n+1)} \left( (\delta\psi^{S^{2}E})(Y)
-\frac{3}{4n-1}g(X, Y) -\frac{1}{2n}\sum_{j}\langle
\nabla_{J_{j}Y} \psi ,\omega_{j}\rangle\right) .
\end{equation}
From (\ref{wanted}) combined with (\ref{alpha}) it is easy to see
that
\begin{equation}\label{codiferentiale}
\alpha =  \frac{1}{4n-1}X,\quad \delta \psi^{S^{2}H} =
-\frac{3}{4n-1}X,\quad \delta\psi^{S^{2}E} = \frac{4n+2}{4n-1}X
\end{equation}
 and thus $\psi^{S^{2}H}$ satisfies the twistor equation
\begin{equation}\label{twistore}
\nabla_{Y}\psi^{S^{2}H} +\frac{1}{4n-1}\sum_{i=1}^{3}\omega_{i}(X,
Y)\omega_{i}=0.
\end{equation}
Using now (\ref{twistore}), an argument like in Lemma 5 of
\cite{lianamax} shows that
\begin{equation}\label{dpsi}
d\psi = -\frac{3}{4n-1}\left( J_{1}X\wedge \omega_{1}+
J_{2}X\wedge\omega_{2}+ J_{3}X\wedge \omega_{3}\right) .
\end{equation}
Substituting (\ref{dpsi}) into the conformal-Killing equation
(\ref{confkil}) we get (\ref{echivalent}), as required.
Conversely, it is clear that any solution $\psi$ of
(\ref{echivalent}) is a conformal-Killing $2$-form and $\delta
\psi =X.$
\end{proof}

We remark that Proposition \ref{corolar} implies that any compatible Killing
$2$-form on $(M, g)$ is parallel, a result previously proved, in the
compact case (and also for higher degree Killing forms),
in \cite{quatk}.

\begin{prop}\label{p1} The map
$$
\psi\rightarrow \left(\psi ,\delta \psi \right)
$$
is an isomorphism from the vector space ${\mathcal C}_{2}(M)$
of compatible conformal-Killing
$2$-forms to the vector space of $\mathcal D$-parallel sections
(where $\mathcal D$ is the connection from Theorem \ref{main2}).
\end{prop}

\begin{proof}
Let $\psi$ be a compatible conformal-Killing $2$-form and $X:=
\delta \psi$ its codifferential. Since $g$ is Einstein, $X$ is a
Killing vector field (see \cite{sem}). When $\nu\neq 0$, $(M, g)$
is irreducible (see Theorem 14.45 of \cite{besse}) and thus
$\nabla X$ is a section of $S^{2}H\oplus S^{2}E$, because $X$ is
Killing (see \cite{kob}, page 246). More generally, when $\nu$ is
arbitrary, a Weitzenb\"{o}ck argument (see \cite{sem}) shows that
\begin{equation}\label{weiz}
\frac{2}{3}\Delta \psi - q(R)\psi +\frac{4(n-1)}{3(4n-1)}d X=0
\end{equation}
where $\Delta =d\delta +\delta d$ is the Laplace operator and
$q(R)$ is a bundle endomorphism of $\Lambda^{2}(M)$, related to
$\Delta$ by $\Delta = \nabla^{*}\nabla + q(R)$, where
$\nabla^{*}\nabla = -\sum_{k=1}^{4n} \nabla^{2}(\psi )(e_{k},
e_{k})$ and $\{ e_{k}\}$ is a local orthonormal frame of $TM.$
Both $\Delta$ and $q(R)$ preserve the irreducible sub-bundles of
$\Lambda^{2}(M)$ and $q(R)$ acts on these sub-bundles by scalar
multiplication (see \cite{w}, Lemma 2.5). Projecting (\ref{weiz})
onto $S^{2}H\otimes \Lambda^{2}_{0}(E)$ and using that $\psi$ a
section of $S^{2}H\oplus S^{2}E$ we get that $\nabla X$ is a
section of $S^{2}H\oplus S^{2}E$, for any $\nu .$

After these preliminary remarks, we now prove that $(\psi , X)$ is
$\mathcal D$-parallel, as follows. From Proposition \ref{corolar},
we know that
$$
{\mathcal D}_{Z}(\psi, X)^{S^{2}E\oplus S^{2}H}=0,\quad\forall Z\in TM.
$$
We need to show that also
\begin{equation}\label{TM}
{\mathcal D}_{Z}(\psi, X)^{TM}=0,\quad\forall Z\in TM.
\end{equation}
For this,  we take the covariant derivative with respect to $Z$ of the
conformal-Killing equation (\ref{echivalent}) and we skew symmetrize
in $Y$ and $Z$. We obtain:
\begin{align*}
[R^{g}_{Z, Y}, \psi] &= \frac{1}{4n-1}\left( \nabla_{Z}X\wedge Y +
\sum_{i=1}^{3}J_{i}\nabla_{Z}X \wedge J_{i}Y -\sum_{i=1}^{3}
\omega_{i}(\nabla_{Z}X,Y)\omega_{i}\right)\\
&-\frac{1}{4n-1}\left( \nabla_{Y}X\wedge Z +
\sum_{i=1}^{3}J_{i}\nabla_{Y}X \wedge J_{i}Z-\sum_{i=1}^{3}\omega_{i}
(\nabla_{Y}X,Z)\omega_{i}\right) .\\
\end{align*}
Applying this relation to a vector $U$, taking the trace over $Y$
and $U$ and applying the result to a vector $V$ we obtain
\begin{equation}\label{sumaadaug}
-4(n+1)(\nabla X)(Z, V)  +4(n+2)(\nabla X)^{S^{2}H}(Z, V) =
(4n-1)g([R^{g}_{Z, e_{k}}\psi ] (e_{k}), V),
\end{equation}
where we used that $\nabla X$ is a section of $S^{2}H\oplus
S^{2}E.$ For any fixed $i\in \{ 1,2,3\}$, replace in
(\ref{sumaadaug}) the pair $(Z, V)$ with $(J_{i}Z, J_{i}V)$ and
sum over $i$. We obtain:
$$
12(n+1)(\nabla X)(Z, V) -4(3n+2)(\nabla X)^{S^{2}H}(Z, V)
= (4n-1) \sum_{i=1}^{3}
g(J_{i}[R^{g}_{J_{i}Z, e_{k}},\psi ](e_{k}), V).
$$
Combining this relation with (\ref{sumaadaug}) we get
\begin{align*}
\nabla_{Z}X  = \frac{4n-1}{16(n+1)}\left( (3n+2)
[R^{g}_{Z,e_{k}},\psi ](e_{k}) + (n+2) \sum_{i=1}^{3}
J_{i}[R^{g}_{J_{i}Z,e_{k}},\psi ](e_{k})\right) .
\end{align*}
We now compute the right hand side of this expression. We first
notice that
\begin{equation}
(3n+2)
[R^{g}_{Z,e_{k}},\psi^{S^{2}E} ](e_{k}) + (n+2) \sum_{i=1}^{3}
J_{i}[R^{g}_{J_{i}Z,e_{k}},\psi^{S^{2}E} ](e_{k})
= -4 [R^{g}_{Z,e_{k}},\psi^{S^{2}E} ](e_{k}).
\end{equation}
On the other hand, from (\ref{rg}),
$$
[ R^{g}_{Z, e_{k}}, \psi^{S^{2}E}](e_{k}) = -\nu (n+1) i_{Z}\psi^{S^{2}E} +
[W^{Q}_{Z,e_{k}},\psi^{S^{2}E}](e_{k}).
$$
Moreover,
\begin{equation}\label{wq}
[W^{Q}_{Z,e_{k}},\psi^{S^{2}E}](e_{k}) =
W^{Q}_{Z, e_{k}}(\psi^{S^{2}E}(e_{k})) = - i_{Z}W^{Q}(\psi^{S^{2}E}),
\end{equation}
where in the first equality (\ref{wq}) we used that $W^{Q}$ is
Ricci-flat and the second equality (\ref{wq}) follows from the
following argument: for any $Y,Z\in TM$,
\begin{align*}
g \left( W^{Q}_{Z, e_{k}}(\psi^{S^{2}E}(e_{k})), Y\right) & =
- \psi^{S^{2}E}(e_{k}, W^{Q}_{Z, e_{k}}(Y)) = -
\langle \psi^{S^{2}E}, e_{k}\wedge W^{Q}_{Z, e_{k}}(Y)
\rangle\\
& = - \langle \psi^{S^{2}E}, W^{Q}_{Z, Y}\rangle = - g( i_{Z}W^{Q}
(\psi^{S^{2}E}), Y),
\end{align*}
for any vector fields $Y$ and $Z$. Therefore,
\begin{align*}
(3n+2) [R^{g}_{Z, e_{k}}, \psi^{S^{2}E}](e_{k}) +(n+2)
\sum_{i=1}^{3}J_{i} [R^{g}_{J_{i}Z, e_{k}}, \psi^{S^{2}E}](e_{k})\\
 =4\left( \nu (n+1) i_{Z}\psi^{S^{2}E}  + i_{Z}W^{Q}(\psi )\right) .
\end{align*}
Similarly, we can prove that
\begin{align*}
(3n+2) [R^{g}_{Z, e_{k}}, \psi^{S^{2}H}](e_{k}) +(n+2)
\sum_{i=1}^{3}J_{i} [R^{g}_{J_{i}Z, e_{k}}, \psi^{S^{2}H}](e_{k})
 = - 8\nu (n+1) i_{Z}\psi^{S^{2}H}.
\end{align*}
We finally obtain
\begin{equation}\label{final}
\nabla_{Z}X = \frac{(4n-1)\nu }{4} i_{Z}\psi^{S^{2}E}
+\frac{4n-1}{4(n+1)}i_{Z}W^{Q}(\psi ) -\frac{\nu
(4n-1)}{2}i_{Z}\psi^{S^{2}H},
\end{equation}
which is equivalent to (\ref{TM}). Our claim follows.
\end{proof}

In order to conclude the proof of Theorem \ref{main2}, we still
need to compute the curvature of the connection $\mathcal D$ and
to show that $\mathcal D$ is flat if and only if $W^{Q}=0.$ This
is done in the following proposition.

\begin{prop}\label{curbura}
The curvature $R^{\mathcal D}$ of the connection $\mathcal D$
defined in Theorem \ref{main2} has
the following expression: for any section $(\psi , X)$ of
$S^{2}H\oplus S^{2}E\oplus TM$ and vector fields $Y, Z\in
{\mathcal X}(M)$,
\begin{align*}
&R^{\mathcal D}_{Y, Z}(\psi , X)^{S^{2}H\oplus S^{2}E}=[W^{Q}_{Y,
Z}, \psi ] - \frac{1}{n+1}\left(
W^{Q}(\psi )\wedge\mathrm{Id}\right)^{S^{2}E}_{Y, Z}\\
&R^{\mathcal D}_{Y, Z}(\psi , X)^{TM}= \frac{n+2}{n+1}
W^{Q}_{Y, Z}X + \frac{4n-1}{4 (n+1)} C(\psi^{S^{2}E})_{Y, Z} ,\\
\end{align*}
where
\begin{equation}
(W^{Q}(\psi )\wedge\mathrm{Id})^{S^{2}E}_{Y, Z}:=\left(
i_{Y}W^{Q}(\psi )\wedge Z - i_{Z}W^{Q}(\psi )\wedge
Y\right)^{S^{2}E}
\end{equation}
and
\begin{equation}\label{adaug3}
C(\psi^{S^{2}E})_{Y, Z}:= i_{Y}\left( \nabla_{Z}W^{Q}\right)
(\psi^{S^{2}E}) -i_{Z}\left( \nabla_{Y}W^{Q}\right)
(\psi^{S^{2}E}).
\end{equation}
In particular, $\mathcal D$ is flat if and only if $W^{Q}=0.$
\end{prop}

\begin{proof}
The $S^{2}H\oplus S^{2}E$ component of $R^{\mathcal
D}$ can be computed as follows. It is straightforward to check that
\begin{align*}
R^{\mathcal D}_{Y, Z}(\psi , X)^{S^{2}H\oplus S^{2}E}&= [R^{g}_{Y,
Z}, \psi ] -\nu \left( i_{Y}(\psi^{S^{2}E})\wedge Z -
i_{Z}(\psi^{S^{2}E})\wedge Y\right)^{S^{2}E}\\
&-\frac{1}{n+1} \left( i_{Y}W^{Q}(\psi )\wedge Z-i_{Z}W^{Q}(\psi
)\wedge Y\right)^{S^{2}E}\\
&+2\nu \left( i_{Y}(\psi^{S^{2}H})\wedge Z -
i_{Z}(\psi^{S^{2}H})\wedge
Y\right)^{S^{2}E}\\
&-\frac{\nu}{2}\sum_{i=1}^{3} \langle \omega_{i},
i_{Y}(\psi^{S^{2}H})\wedge Z - i_{Z}(\psi^{S^{2}H})\wedge
Y\rangle\omega_{i}.
\end{align*}

On the other hand, the following equalities hold:
for any vector fields $Y, Z\in {\mathcal X}(M)$,
\begin{align*}
&[R^{g}_{Y, Z}, \psi^{S^{2}H}]=\frac{\nu}{2}\sum_{i=1}^{3} \langle
\omega_{i},i_{Y}(\psi^{S^{2}H})\wedge Z -
i_{Z}(\psi^{S^{2}H})\wedge Y\rangle \omega_{i};\\
&[R^{g}_{Y,Z}, \psi^{S^{2}E}]= \nu \left( i_{Y}(\psi^{S^{2}E})
\wedge Z - i_{Z}(\psi^{S^{2}E})\wedge
Y\right)^{S^{2}E} + [W^{Q}_{Y,Z},\psi^{S^{2}E}];\\
&\left( i_{Y}(\psi^{S^{2}H})\wedge Z -
i_{Z}(\psi^{S^{2}H})\wedge Y\right)^{S^{2}E}=0.
\end{align*}
From these relations we get
$$
R^{\mathcal D}_{Y,Z} (\psi , X)^{S^{2}H\oplus S^{2}E} =
[W^{Q}_{Y,Z},\psi ] -\frac{1}{n+1} \left( W^{Q}(\psi )
\wedge \mathrm{Id}\right)^{S^{2}E}_{Y,Z}
$$
as required.  The $TM$ component of $R^{\mathcal D}$ can be
computed in a similar way. It is obvious now that $\mathcal D$ is
flat if and only if $W^{Q}=0.$ Our claim follows.
\end{proof}

The proof of Theorem \ref{main2} is now completed. We end this
Section with the following result, which is a consequence of the
curvature computation from Proposition \ref{curbura} and will be
used in the proof of Proposition \ref{sp} from the next section.
We remark that relations similar to (\ref{w}) and (\ref{w1}) hold
also in the K\"{a}hler setting, with $u$ replaced by (the
trace-free part of) a Hamiltonian $2$-form and the
quaternionic-Weyl tensor $W^{Q}$ replaced by the Bochner tensor of
the K\"{a}hler manifold (see Proposition 9 of \cite{gaud}).

\begin{prop}\label{1}
Let $(M, g)$ be a quaternionic-K\"{a}hler manifold of dimension
$4n\geq 8$. Let $\psi$ be a compatible conformal-Killing $2$-form on
$(M, g)$, $u:= \psi^{S^{2}E}$ its $S^{2}E$-part and $X:=\delta
\psi .$ Then
\begin{equation}\label{w}
W^{Q}_{V, X}= \frac{4n-1}{4(n+2)}(\nabla_{V}W^{Q})(u) =
\frac{4n-1}{4(n+1)}\nabla_{V}\left( W^{Q}(u)\right) ,\quad\forall
V\in TM.
\end{equation}
Moreover,
\begin{equation}\label{w1}
[W^{Q}(u), v ] = (n+1) [W^{Q}(v ), u],\quad\forall v\in S^{2}E
\end{equation}
and
\begin{equation}\label{w5}
[W^{Q}(u), u ] = [(\nabla X)^{S^{2}E}, u ] =0.
\end{equation}
\end{prop}

\begin{proof} From Propositions \ref{p1} and \ref{curbura},
$(\psi, X)$ is $\mathcal D$-parallel and
\begin{equation}\label{rel}
(n+2)W^{Q}_{Y ,Z}X +(n-\frac{1}{4})C(u)_{Y, Z}=0,\quad\forall Y,
Z\in TM.
\end{equation}
On the other hand, from the definition of the tensor $C$, for any $Y, Z,
V\in TM$,
\begin{align*}
C(u)_{Y, Z}(V)&= \langle (\nabla_{Z}R^{g})(u), Y\wedge V\rangle
-\langle (\nabla_{Y}R^{g})(u), Z\wedge V\rangle\\
&=\langle (\nabla_{Z}R^{g})(Y\wedge V), u\rangle -\langle
(\nabla_{Y}R^{g})(Z\wedge V), u\rangle\\
& =\langle (\nabla_{V}R^{g})(Y\wedge Z), u\rangle
\end{align*}
since $\nabla_{Z}R^{g}\in\mathrm{End}(\Lambda^{2}(M))$
is symmetric and $R^{g}$ satisfies the second Bianchi identity.
Relation (\ref{rel}) becomes
\begin{equation}\label{nplus}
(n+2) g( W^{Q}_{Y, Z}X, V) +(n-\frac{1}{4}) \langle
(\nabla_{V}R^{g})(Y\wedge Z), u\rangle =0
\end{equation}
and implies the first relation
(\ref{w}), because both $W^{Q}$ and $\nabla_{V}R^{g}$ are symmetric
endomorphisms of $\Lambda^{2}M$. The second relation (\ref{w})
follows from the first, by using
$$
(\nabla_{V}W^{Q})(u) = \nabla_{V} \left( W^{Q}(u)\right) - W^{Q}
(\nabla_{V}u) = \nabla_{V} \left( W^{Q}(u)\right) -\frac{4}{4n-1}
W^{Q}_{X, V},
$$
because
$$
\nabla_{V}u = \frac{1}{4n-1}\left( X\wedge V +\sum_{i=1}^{3}
J_{i}X\wedge J_{i}V\right)
$$
and $W^{Q}$ is $J_{i}$-invariant (as an $\mathrm{End}(TM)$-valued form).
Relation (\ref{w}) is proved. In order to prove relation
(\ref{w1}) note first that for any $v\in S^{2}E$,
\begin{equation}\label{put}
(W^{Q}(u)\wedge \mathrm{Id})^{S^{2}E}(v) = [W^{Q}(u), v]
\end{equation}
(relation (\ref{put}) is obtained by writing $v =
\frac{1}{2}\sum_{k} e_{k}\wedge v (e_{k})$ with respect to a local
orthonormal frame $\{ e_{k}\}$ of $TM$ and using the definition of
$(W^{Q}(u)\wedge \mathrm{Id})^{S^{2}E}$ and the $J_{i}$-invariance
of $W^{Q}(u)$ and $v$). Relation (\ref{w1}) follows now from
(\ref{put}), together with
$$
[W^{Q}_{Y, Z}, u]  = \frac{1}{n+1} (W^{Q}(u)\wedge
\mathrm{Id})^{S^{2}E}_{Y, Z},\quad\forall Y, Z\in TM
$$
(see Proposition \ref{curbura}). It remains to prove (\ref{w5}).
This is a consequence of (\ref{w1}) and
$$
(\nabla X)^{S^{2}E} = \frac{(4n-1)\nu}{4}u +\frac{4n-1}{4(n+1)} W^{Q}(u),
$$
(which follows from ${\mathcal D}(\psi , X)=0$).
\end{proof}

\section{Applications of our main result}\label{applications}

In this Section we develop several application of Theorem \ref{main2}.

\subsection{The dimension of the space ${\mathcal C}_{2}(M)$}

As a first application of Theorem \ref{main2} we determine a
sharp estimate for the dimension of the vector space
${\mathcal C}_{2}(M)$
of compatible
conformal-Killing $2$-forms on a quaternionic-K\"{a}hler
manifold $(M, g)$. It is known that on an arbitrary
Riemannian manifold (not necessarily compact) the space of
conformal-Killing forms (of any degree) is finite-dimensional and an upper bound,
which is realized on the standard sphere, was found in \cite{sem}.
For compatible conformal-Killing $2$-forms on quaternionic-K\"{a}hler
manifolds there is the following similar result:

\begin{cor}\label{dim} Let $(M, g)$ be a quaternionic-K\"{a}hler
manifold of dimension $4n\geq 8$. Then
\begin{equation}\label{dimbound}
\mathrm{dim}{\mathcal C}_{2}(M)\leq (n+1)(2n+3).
\end{equation}
Equality holds on the standard $\mathbb{H}P^{n}.$
\end{cor}

\begin{proof}
Notice that $(n+1)(2n+3)$ is the rank of the bundle $S^{2}H\oplus
S^{2}E\oplus TM$ on which the connection $\mathcal D$ is defined.
Therefore, inequality (\ref{dimbound}) follows from Theorem
\ref{main2}. It remains to show that equality holds on the
quaternionic projective space $\mathbb{H}P^{n}$, with its standard
quaternionic-K\"{a}hler structure. As proved in \cite{lianamax},
any conformal-Killing $2$-form on $\mathbb{H}P^{n}$ is compatible
and the co-differential $\delta$ defines an isomorphism from
${\mathcal C}_{2}(\mathbb{H}P^{n})$ onto the space of Killing
vector fields. The latter has dimension $(n+1)(2n+3).$ Our claim
follows.
\end{proof}

\subsection{The case when $\nu\neq 0$}

Let $(M,g)$ be quaternionic K\"{a}hler manifold (as usual,
connected and of dimension $4n\geq 8$). In this section we assume
that the scalar curvature of $(M, g)$ is non-zero. Our main
results in this setting are Propositions \ref{sp} and \ref{hamkil}
(see below).

\begin{prop}\label{sp}
If $(M, g)$ has a non-parallel compatible conformal-Killing
$2$-form, then the holonomy group of $(M,g)$ is
$\mathrm{Sp}(1)\mathrm{Sp}(n).$
\end{prop}

\begin{proof}
Let $\psi$ be a non-parallel compatible conformal-Killing $2$-form
on $(M, g)$. We will show that the holonomy algebra
$\mathrm{hol}(M,g)$ of $(M, g)$ coincides with
$\mathrm{sp}(1)\oplus\mathrm{sp}(n).$ From Proposition
\ref{corolar}, $X:=\delta \psi$ is non-trivial. Recall that
$\mathrm{hol}(M, g)$ contains $(\nabla_{V}R^{g})_{Y, Z}$ for any
$Y, Z, V\in TM$ (see Chapter 10 of \cite{besse}). Therefore,
relation (\ref{w}) implies that $W^{Q}_{X, V}\in\mathrm{hol}(M,
g)$ for any $V$. Using this fact, the proof of our claim follows
like in Lemma 17 of \cite{lianamax}. For completeness of our
exposition, we include the argument. Since $R^{g}_{Y, Z}$ belongs
to the holonomy algebra as well, and $\mathrm{sp}(1)\subset
\mathrm{hol}(M,g)$ (since the scalar curvature is non-zero, see
Lemma 14.46 of \cite{besse}), we deduce that $(R^{g}_{Y,
Z})^{S^{2}E}$ belongs to $\mathrm{hol}(M, g)$, for any $Y, Z\in
TM.$ It follows that
\begin{equation}\label{a01}
(R^{g}_{X, V})^{S^{2}E} - W^{Q}_{X, V}= -\nu (X\land
V)^{S^{2}E}\in\mathrm{hol}(M,g),\quad\forall V
\end{equation}
and
\begin{equation}\label{a02}
(R^{g}_{J_{i}X, V})^{S^{2}E}- W^{Q}_{J_{i}X, V}= -\nu (J_{i}X\land
V)^{S^{2}E}\in\mathrm{hol}(M,g),\quad\forall V,
\end{equation}
where  $\{ J_{1}, J_{2}, J_{3}\}$ is a local admissible basis of $Q$.
We just proved that if $Y$ or $U$ belong to
${\mathcal V}:= \mathrm{Span}\{ X, J_{1}X, J_{2}X, J_{3}X\}$, then
$(Y\land U)^{S^{2}E}$ belongs to $\mathrm{hol}(M,g).$ It remains
to show that $(Y\land U)^{S^{2}E}$ belongs to the holonomy algebra
when both $Y$ and $U$ are orthogonal to ${\mathcal V} .$ Take
such two tangent vectors $Y$ and $U$. Since both
$(X\land Y)^{S^{2}E}$ and $(X\land U)^{S^{2}E}$ belong to
$\mathrm{hol}(M,g)$, also their Lie bracket, which is equal to
$$
[(X\land Y)^{S^{2}E}, (X\land U)^{S^{2}E}] =
\frac{1}{16}\sum_{i,j=1}^{3} g(J_{i}Y, J_{j}U)J_{i}X\land J_{j}X
+\frac{1}{4} g(X, X) (Y\land U)^{S^{2}E},
$$
belongs to $\mathrm{hol}(M,g)$, as well as the $S^{2}E$-part of
this Lie bracket. Using (\ref{a01}) and (\ref{a02}) we get that
$(Y\wedge U)^{S^{2}E}\in\mathrm{hol}(M,g).$ Our claim follows.
\end{proof}

\begin{cor}\label{adaugat} One of the following two statements holds:\\

i) either any compatible conformal-Killing $2$-form on $(M, g)$ is parallel;\\

ii) or any parallel $2$-form is trivial. In particular, the
codifferential
$$
\delta : {\mathcal C}_{2}(M)\rightarrow \mathrm{isom}(M, g)
$$
is injective.
\end{cor}

\begin{proof}
Suppose that $(M, g)$ admits a non-parallel compatible
conformal-Killing $2$-form. Then, from Proposition \ref{sp}, the
holonomy group of $(M, g)$ is $\mathrm{Sp}(1)\mathrm{Sp}(n)$ and
therefore $(M, g)$ does not admit non-trivial parallel $2$-forms
(see e.g. \cite{besse}, page 306). From Proposition \ref{corolar},
any compatible Killing $2$-form is parallel. Thus, the
codifferential $\delta$ defined on ${\mathcal C}_{2}(M)$ is
injective.

\end{proof}

\begin{prop}\label{hamkil}
The map
\begin{equation}\label{map}
{\mathcal C}_{2}(M)\ni \psi \rightarrow u:= \psi^{S^{2}E}
\end{equation}
is an isomorphism from the vector space ${\mathcal C}_{2}(M)$ of
compatible conformal-Killing $2$-forms on $(M, g)$ to the vector
space of (real) sections of $S^{2}E$ which satisfy
\begin{equation}\label{hamiltonian}
\nabla_{Y}u = \frac{1}{4n-1}\left( X\wedge Y
+\sum_{i=1}^{3}J_{i}X\wedge J_{i}Y\right) \quad\forall Y\in TM,
\end{equation}
where $X\in {\mathcal X}(M)$ is a vector field on $M$ (necessarily
equal to $\frac{4n-1}{4n+2}\delta u$) and  $\{ J_{1}, J_{2},
J_{3}\}$ is an admissible basis of $Q$. The inverse is the map
\begin{equation}\label{inverse}
u \rightarrow \psi := u -\frac{1}{(2n+1)\nu} (\nabla \delta u
)^{S^{2}H}.
\end{equation}

\end{prop}

\begin{proof} Let $\psi$
be a compatible conformal-Killing $2$-form and $u:= \psi^{S^{2}E}$
its $S^{2}E$-component. Projecting (\ref{echivalent}) onto
$S^{2}E$ we obtain (\ref{hamiltonian}). Moreover, since $\nu\neq
0$,
$$
\psi^{S^{2}H} = -\frac{1}{(2n+1)\nu} (\nabla \delta u )^{S^{2}H}
$$
because both sides are solutions of the twistor equation, with
equal codifferentials (recall relation (\ref{codiferentiale}) and
our comments from Section \ref{sectiunepen}). It follows that
\begin{equation}\label{rec}
\psi = u -\frac{1}{(2n+1)\nu} (\nabla \delta u)^{S^{2}H}.
\end{equation}
It remains to show that the map (\ref{map}) is onto the space of
solutions of (\ref{hamiltonian}), i.e. any section $u$ of
$S^{2}E$, which is a solution of (\ref{hamiltonian}), is the
$S^{2}E$-part of a compatible conformal-Killing $2$-form. For
this, let $u$ be a solution of (\ref{hamiltonian}), where $X$ is a
vector field (necessarily equal to $\frac{4n-1}{2(2n+1)}\delta
u$). We will show that $X$ is Killing, or, equivalently, $X$ is
quaternionic (being divergence-free and $\nu\neq 0$, see Section
\ref{sectiunepen}). Taking the covariant derivative of
(\ref{hamiltonian}) with respect to $Z$ and skew symmetrizing in
$Y$ and $Z$ we obtain
\begin{equation}\label{curv-u}
[R^{g}_{Z, Y}, u] = \frac{4}{4n-1}\left( \nabla_{Z}X\wedge Y -
\nabla_{Y}X\wedge Z\right)^{S^{2}E}
\end{equation}
 On the other hand, since $u$ is a section of $S^{2}E$,
\begin{equation}\label{doi}
[R^{g}_{Z, Y}, u] = [(R^{g}_{Z, Y})^{S^{2}E}, u] = [R^{g}_{J_{1}Z,
J_{1}Y}, u]
\end{equation}
so the right hand side of (\ref{curv-u}) remains unchanged if we
replace $(Z, Y)$ by $(J_{1}Z, J_{1}Y).$ Defining
$$
S(Z, X):= [\nabla X, J_{1}](Z)= \nabla_{J_{1}Z}X -
J_{1}\nabla_{Z}X,\quad\forall Z\in {\mathcal X}(M)
$$
we get:
\begin{align*}
S(Z, X)\wedge J_{1}Y - J_{1}S(Z, X)\wedge Y - J_{2}S(Z, X)\wedge
J_{3}Y +
J_{3}S(Z, X)\wedge J_{2}Y\\
- S(Y, X)\wedge J_{1}Z + J_{1}S(Y, X)\wedge Z + J_{2}S(Y, X)\wedge
J_{3}Z
-J_{3}S(Y, X)\wedge J_{2}Z=0.\\
\end{align*}
Applying this relation to a vector $U$ and taking the trace over
$Y$ and $U$ we obtain
\begin{equation}\label{rel-quat}
S(Z, X) = \frac{1}{4n}\sum_{k=1}^{4n} \left(g(S(e_{k},
X),J_{2}e_{k})J_{2}Z + g(S(e_{k}, X),J_{3} e_{k}) J_{3}Z\right)
,\quad\forall Z\in TM,
\end{equation}
where $\{ e_{k}\}$ is a local orthonormal frame of $TM$. Relation
(\ref{rel-quat}) implies that $X$ is quaternionic. Being
quaternionic and divergence-free (and $\nu\neq 0$), $X$ is
Killing. Using the Koszul formula for $X$ and (\ref{hamiltonian}),
it can be checked that
$$
\psi := u -\frac{2}{\nu (4n-1)} (\nabla X)^{S^{2}H}
$$
satisfies (\ref{echivalent}) and hence is conformal-Killing.
Obviously, $\psi$ is a compatible $2$-form and its $S^{2}E$-part
coincides with $u$. Our claim follows.

\end{proof}

\section{A bracket on conformal-Killing $2$-forms}\label{bracket-section}

In this final Section we define a skew-symmetric multiplication on
the space of conformal-Killing $2$-forms on a
quaternionic-K\"{a}hler manifold and we study its properties in
relation with the subspace of compatible conformal-Killing
$2$-forms. Such a multiplication can be defined in the more
general setting of Einstein manifolds, as follows.

\begin{prop}\label{general} If $\psi_{1}$ and $\psi_{2}$ are
conformal-Killing $2$-forms on an Einstein manifold, then
\begin{equation}\label{111}
[\psi_{1}, \psi_{2}] := \frac{1}{2} \left(
L_{\delta\psi_{1}}\psi_{2} - L_{\delta\psi_{2}}\psi_{1}\right)
\end{equation}
is also conformal-Killing and

\begin{equation}\label{co-dif}
\delta [\psi_{1}, \psi_{2}] =
[\delta\psi_{1}, \delta\psi_{2}].
\end{equation}
\end{prop}

\begin{proof} It is easy to check that the Lie
derivative of a conformal-Killing form $\psi$ (of any degree) with
respect to a Killing vector field $X$ is also conformal-Killing,
with codifferential $L_{X}(\delta \psi ).$  Since on an Einstein
manifold the codifferential of a conformal-Killing $2$-form is a
Killing vector field (see \cite{sem}), the bracket $[\psi_{1},
\psi_{2}]$ of two conformal-Killing $2$-forms $\psi_{1}$ and
$\psi_{2}$, as defined in (\ref{111}), is also conformal-Killing.
Relation (\ref{co-dif}) is straightforward.
\end{proof}

Since any quaternionic-K\"{a}hler manifold is Einstein,
Proposition \ref{general} implies that (\ref{111}) is a
skew-symmetric multiplication on the space of conformal-Killing
$2$-forms on any quaternionic-K\"{a}hler manifold.

\begin{cor}\label{particular} i) The bracket (\ref{111}) preserves the subspace ${\mathcal
C}_{2}(M)$ of compatible conformal-Killings $2$-forms on a
quaternionic-K\"{a}hler manifold $(M,g).$\\

ii) Assume that $(M,g)$ is not Ricci-flat. Then $({\mathcal
C}_{2}(M), [\cdot ,\cdot ])$ is a Lie algebra and the
codifferential
\begin{equation}\label{codi}
\delta : {\mathcal C}_{2}(M)\rightarrow \mathrm{isom}(M, g)
\end{equation}
is a Lie algebra homomorphism.

\end{cor}

\begin{proof} Let $\psi_{1}$ and $\psi_{2}$ be two
compatible conformal-Killing $2$-forms, with codifferentials
$X_{1}$ and $X_{2}$, which are Killing vector fields. From
Proposition \ref{general} we know that $[\psi_{1}, \psi_{2}]$ is
conformal-Killing and we need to show that it is a section of
$S^{2}H\oplus S^{2}E.$ As already mentioned before, $\nabla X_{i}$
($1\leq i\leq 2$) are sections of $S^{2}H\oplus S^{2}E$. It
follows that
$$
L_{X_{1}}\psi_{2} = \nabla_{X_{1}}\psi_{2} - [\nabla X_{1},
\psi_{2} ]
$$
is a section of $S^{2}H\oplus S^{2}E$. A similar argument shows
that $L_{X_{2}}\psi_{1}$ is also a section of $S^{2}H\oplus
S^{2}E$. Thus $[\psi_{1}, \psi_{2}]$ is a compatible
conformal-Killing $2$-form. This proves the first claim. For the
second claim, let $\psi_{1}, \psi_{2}, \psi_{3}\in {\mathcal
C}_{2}(M)$ and define
$$
\psi := [[\psi_{1}, \psi_{2}], \psi_{3}] +[[\psi_{3}, \psi_{1}],
\psi_{2}] + [[\psi_{2}, \psi_{3}], \psi_{1}].
$$
From Proposition \ref{general}, $\psi$ is a Killing $2$-form.
Being compatible, it is parallel. Recall now, from Corollary
\ref{adaugat}, that on a quaternionic-K\"{a}hler manifold with
non-zero scalar curvature, either any parallel $2$-form is trivial
or any compatible conformal-Killing $2$-form is parallel. In both
cases, $\psi =0$ (note that the bracket $[\psi_{i}, \psi_{j}]$ is
zero when both $\psi_{i}$ and $\psi_{j}$ are parallel). Thus
$({\mathcal C}_{2}(M), [\cdot , \cdot ])$ is a Lie algebra. From
Proposition \ref{general}, the map (\ref{codi}) is a Lie algebra
homomorphism.

\end{proof}

\textbf{Acknowledgements:} I am grateful to Vestislav Apostolov
for suggesting me to study compatible conformal-Killing $2$-forms
on quaternionic K\"{a}hler manifolds. I also thank Paul Gauduchon
for pertinent remarks on a first version of this paper. Useful
discussions with Jose Figueroa-O'Farrill about a Lie algebra
structure on the space of conformal-Killing $2$-forms are also
acknowledged. This work is supported by a CNCSIS grant IDEI
``Structuri geometrice pe varietati diferentiable'' code
1187/2008.

\end{document}